\documentclass[12pt]{amsart}

\usepackage[all]{xy}
\usepackage{graphics,mathabx,bm}
\usepackage{color}
\usepackage[T1]{fontenc}
\usepackage{parskip}
\usepackage[colorlinks]{hyperref}
\usepackage{enumitem}
\usepackage{comment}

\usepackage{tikz,tikz-cd}
\usetikzlibrary{decorations.markings}

\usepackage[colorinlistoftodos,prependcaption]{todonotes}

\newcommand{\B}[1]{{\mathbf #1}}

\newtheorem{theorem}{Theorem}[section]
\newtheorem{theorem*}{Theorem}
\newtheorem{lemma}[theorem]{Lemma}
\newtheorem{proposition}[theorem]{Proposition}
\newtheorem{corollary}[theorem]{Corollary}

\newtheorem{problem}[theorem]{Problem}
\newtheorem*{question*}{Question}

\theoremstyle{definition}
\newtheorem{example}[theorem]{Example}

\newtheorem*{remark*}{Remark}
\newtheorem*{remarks*}{Remarks}
\newtheorem*{corollary*}{Corollary}

\numberwithin{figure}{section}
\numberwithin{table}{section}
\numberwithin{equation}{section}

\def\B{\mathbf}

\newcommand{\OP}{\operatorname}

\begin{document}

\title[Aut-invariant word norm on graph products]{Aut-invariant word norm on right angled Artin and right angled Coxeter groups}
\author{Micha\l\ Marcinkowski}
\address{Regensburg Universit{\"a}t \& Uniwersytet Wroc\l awski}
\email{marcinkow@math.uni.wroc.pl}
\keywords{Artin groups, Coxeter groups, quasimorphisms, Aut-invariant norm}
\subjclass[2010]{51,20}

\begin{abstract} 
We prove that the $\OP{Aut}$-invariant word norm on right angled Artin and right angled Coxeter groups is unbounded 
(except in few special cases). 
To prove unboundedness we exhibit certain characteristic subgroups. 
This allows us to find unbounded quasimorphisms which are Lipschitz with respect to
the $\OP{Aut}$-invariant word norm. 
\end{abstract}

\maketitle



Let $W$ be a right angled Artin group or a right angled Coxeter group. 
The full automorphism group, $\OP{Aut}(W)$, is a classical object studied both from combinatorial 
and geometrical points of view (\cite{gpr}, \cite{crsv} and references therein, for right angled Artin groups we refer to a survey \cite{v}). 
Examples of $\OP{Aut}(W)$ include $\OP{Aut}(F_n)$ and $\OP{Aut}(Z^n) = GL_n(Z)$. 
If~$W$~is a right angled Artin group, then $\OP{Aut}(W)$ interpolates between these two groups. 

In this paper we study the standard action of $\OP{Aut}(W)$ on $W$ in relation to the $\OP{Aut}$-invariant word norm on $W$.
This word norm is defined analogously to the standard word norm 
where one requires invariance under the full automorphism group. 
For free groups and surface groups such norms were already studied in \cite{bm,bst} and it is worth to note, 
that they have natural interpretations. 
In the case of free groups, the $\OP{Aut}$-invariant norm is the 
word norm associated to the set of all primitive elements, i.e., elements which can be extended to a free basis.  
For surface groups, it is the word norm associated to the set of all based simple loops. 

The purpose of this paper it to prove Theorem \ref{t:main2}, which classifies graph products of abelian groups
with bounded $\OP{Aut}$-invariant norm. 
Note that right angled Artin groups and right angled Coxeter groups are special cases of graph products of abelian groups. 

Before stating our result let us discuss few basic examples that motivated this note.
It was known that the free groups $F_n$, for $n=1,2,\ldots$, have unbounded $\OP{Aut}$-invariant norms \cite{bst}. 
On the other hand, $Z^n$ has bounded $\OP{Aut}$-invariant norm for $n>2$, see Lemma \ref{l:oneclass}.
A question which arises is: what right angled Artin groups have bounded $\OP{Aut}$-invariant norm? 
The same question we ask for right angled Coxeter groups.

Denote by $C_n$ the cyclic group of order $n$. 
We have that the dihedral group $D_\infty = C_2 \ast C_2$ has bounded $\OP{Aut}$-invariant norm, but already
the group $PSL(Z) = C_2 \ast C_3$ has unbounded $\OP{Aut}$-invariant norm, see Lemma \ref{l:dihedral} and Lemma \ref{l:rolli}. 
The natural class to study all groups mentioned above is the class of graph products of abelian groups. 

Let us now state the main result in the form specified to right angled Artin groups and right angled Coxeter groups. 

\begin{theorem}\label{t:main}
Let $W$ be a right angled Artin or a right angled Coxeter group. 
The $\OP{Aut}$-invariant word norm on $W$ is bounded if and only if 
$W=Z^n,~n>1$ or $W = D_{\infty}^n \times C_2^m$, where $D_{\infty}$ is the infinite dihedral group 
and $C_2$ is the group of order 2. 
\end{theorem}

The outline of the proof is as follows. 
First we determine which special subgroups of $W$ are (almost) characteristic.
Passing to quotients allows us to reduce the problem to simpler Artin and Coxeter groups. 
Unboundedness is proven by finding unbounded quasimorphism on~$W$~that is Lipschitz 
with respect to the $\OP{Aut}$-invariant word norm. 
It~follows that if the $\OP{Aut}$-invariant word norm is unbounded, there exist 
elements $w \in W$ that are undistorted in this norm, 
i.e., the $\OP{Aut}$-invariant word norm of $w^n$ growths linearly with $n$.

\textbf{Acknowledgements}. 
The author was supported by SFB 1085 ``Higher Invariants'' funded by the Deutsche Forschungsgemeinschaft DFG.
Part of this work was conducted during the author's stay at the Ben Gurion University, 
supported by GIF YOUNG GRANT \#I-2419-304.6/2016. 
The author would like to thank Michael Brandenbursky for his support and helpful discussions. 

\section{Aut-invariant word norm and quasimorphisms}

In this paper we are interested in the $\OP{Aut}$-invariant word norms. 
However, to prove our main result, we need to introduce a slightly broader class of norms.

We say that two norms $|\mathord{\cdot}|_1$ and $|\mathord{\cdot}|_2$ are \textbf{bi-Lipschitz equivalent}, if there exists
$C \in \B R$ such that $C^{-1}|\mathord{\cdot}|_2 \leqslant |\mathord{\cdot}|_1 \leqslant C|\mathord{\cdot}|_2$.

Suppose $G$ is a finitely generated group and $H \leqslant \OP{Aut}(G)$.
Let $S \subseteq G$ be a finite set such that $\bar{S} = HS = \{ \psi(s)~|~\psi\in H, s\in S\}$ generates $G$. 
By $|\mathord{\cdot}|_{\bar{S}}$ we denote the word norm associated to the subset $\bar{S}$, i.e.:
$$
|x|_{\bar{S}} = min\{n~|~x = s_1\ldots s_n,~\text{where}~s_i\in \bar{S}~\text{for each}~i\}.
$$
This norm is $H$-invariant, i.e., $|x|_H = |\psi(x)|_H$ for $\psi \in H$. 
We define the $\boldsymbol H$\textbf{-invariant word norm} on $G$ to be the bi-Lipschitz equivalence class of $|\mathord{\cdot}|_{\bar{S}}$.
It is straightforward to show that this definition does not depend on the choice of $S$.
Usually we refer to the $H$-invariant norm as to a genuine norm on $G$, 
implicitly choosing a representative of the bi-Lipschitz equivalence class. 

\begin{example}
\begin{enumerate}
\item If $H=\{e\}$, then $|\mathord{\cdot}|_H$ is the standard word norm on a finitely generated group.
\item If $H = \OP{Inn}(G)$, then $|\mathord{\cdot}|_H$ is the bi-invariant word norm. 
If~$G$~is a Coxeter group, the bi-invariant norm is equivalent to the reflection norm \cite{gk,bgkm}.
\item Let $H = \OP{Aut}(G)$. If $G$ is a free group, then  $|\mathord{\cdot}|_H$ is equivalent to the primitive norm.
If $G$ is a surface group, then $|\mathord{\cdot}|_H$ is equivalent to the simple loops norm \cite{bm}.
\end{enumerate}
\end{example}

We are interested in the special case when $H = \OP{Aut}(G)$. 
For simplicity we write $|\mathord{\cdot}|_{Aut}= |\mathord{\cdot}|_{Aut(G)}$ when the group $G$ is understood. 
The norm $|\mathord{\cdot}|_{Aut}$ is called the $\OP{\bf Aut}$-\textbf{invariant word norm}.
For more detailed discussion on $H$-invariant word norms we refer to \cite{bm}.

Let $W$ be a right angled Artin or Coxeter group. 
In what follows, we need to focus not on $|\mathord{\cdot}|_{Aut}$ directly, but on 
$|\mathord{\cdot}|_{H_0}$ where $H_0$ is a certain finite index subgroup of $\OP{Aut}(W)$. 
Due to the following lemma, these norms are equivalent. 
\begin{lemma}\label{l:equivalent}
	Let $H_0 \leqslant H \leqslant \OP{Aut}(G)$ and let $H_0$ be finite index in $H$. 
	Then $|\mathord{\cdot}|_H$ and $|\mathord{\cdot}|_{H_0}$ are equivalent.
\end{lemma}

\begin{proof}
	Let $S$ be a finite subset of $G$ such that $\bar{S}=HS$ generates $G$. 
	Since $\#(\bar{S}/H_0) = \#(H/H_0)\#(\bar{S}/H)$, there exists a finite subset $S'$ of $G$ satisfying 
	$\bar{S}=H_0S'$. Thus $|\mathord{\cdot}|_{\bar{S}}$ can be used to define 
	both $|\mathord{\cdot}|_H$ and $|\mathord{\cdot}|_{H_0}$. 
\end{proof}

Let us recall a notion of a quasimorphism. 
A function $q\colon G\to\B R$ is called a \textbf{quasimorphism} 
if $|q(a)-q(ab)+q(b)|<D$ for some $D \in \B R$ and all $a,b\in G$. 
A quasimorphism is \textbf{homogeneous} if $q(x^n)~=~nq(x)$ for all $n \in\B Z$ and all $x \in G$.
Homogeneous quasimorphisms are constant on conjugacy classes, i.e., $q(x) = q(yxy^{-1})$ for all $x,y\in G$.
The homogenisation of $q$ is defined by 
$$
\bar{q}(x) = \lim_{s\to \infty} \frac{q(x^s)}{s}.
$$

The function $\bar{q}$ is a homogeneous quasimorphism. 
Moreover, there exists $C \in \B R$ such that $|q(x) - \bar{q}(x)| < C$ for every $x \in G$.
We refer to \cite{scl} for further details.

An element $x \in G$ is \textbf{undistorted} in a norm $|\mathord{\cdot}|$ 
if there exists a positive real number $C$ such that $|x^n| > Cn$ for all $n \in \B Z$.
Otherwise $x$ is \textbf{distorted}. 
If $x$ is undistorted in the $\OP{Aut}$-invariant word norm, we call it
$\OP{\bf Aut}$-\textbf{undistorted}. Otherwise $x$ is $\OP{\bf Aut}$-\textbf{distorted}. 

The relation between quasimorphisms and $H$-invariant word norms is explained in the following Lemma. 

\begin{lemma}[\cite{bm}, Lemma 1.5]\label{l:qm.undist}
	 
	 Let $G$ be a group and let $H \leqslant \OP{Aut}(G)$. 
	 Let $S \subseteq G$ be a finite set such that $\bar{S} = HS$ generates $G$.
	 If there exists a homogeneous quasimorphism $q\colon G \to\B R$ bounded on $\bar{S}$ 
	 and such that $q(x) \neq 0$, then $x$ is undistorted in $|\mathord{\cdot}|_H$.
 \end{lemma} 

 \section{Graph products of abelian groups}\label{s:graph.prod}

Right angled Coxeter groups and right angled Artin groups are graph products of abelian groups. 
In Section \ref{s:proof} we prove a more general version of Theorem \ref{t:main} 
which involves all graph products of finitely generated abelian groups. 
This generality allows us to bring right angled Artin and Coxeter groups in a common framework. 

In this section we collect results on graph products of abelian groups and their automorphisms groups.
We mainly follow \cite{cg}.  

Let $\Gamma$ be a simplicial graph, $V$ its vertex set and $E$ its edge set.
Let $\{G_v\}_{v\in V}$ be a collection of finitely generated abelian groups. 
The graph product given by $\Gamma$ and $\{G_v\}_{v\in V}$ is the group 
$$
W(\Gamma,\{G_v\}) = \Asterisk \{G_v\}_{v\in V}/N, 
$$
where $\Asterisk \{G_v\}_{v \in V}$ is the free product of groups $G_v$, 
and $N$ is the normal subgroup generated by $[G_x,G_y]$, $(x,y) \in E$.

Let $W$ be a graph product of finitely generated abelian groups.
A~cyclic group is \textbf{primary} if its order is a power of a prime number.
It is easy to see that there exists a graph $\Gamma$ and a collection $\{G_v\}_{v\in V}$ such that
$W$ is isomorphic to $W(\Gamma,\{G_v\})$ and each $G_v$ is primary or infinite cyclic. 
Thus without loss of generality we may always assume that $W$ is a graph product where each $G_v$ is 
either primary cyclic or infinite cyclic. 

Throughout this section and Section \ref{s:aut^0},  $\Gamma$ is a finite simplicial graph, 
$V$~its vertex set and $\{G_v\}_{v \in V}$ a collection of cyclic groups, each primary or infinite. 
Let $W_\Gamma = W(\Gamma,\{G_v\})$.
For each $v \in V$ we fix a generator of~$G_v$. Abusing notation, we call this generator again by $v$
making no distinction between the generator of $G_v$ and the vertex $v$.

Given a vertex $v \in V$, by $St(v) = \{ w \in V~|~[w,v] = e\}$ denote the star of $v$ in $\Gamma$ and
by $Lk(v) = St(v) - \{v\}$ the link of $v$ in $\Gamma$.
Recall that a preorder is a reflexive and transitive relation. 
We define two preorders on $V$ by:

\begin{center}
\begin{tabular}{ l c r l}
$v \leqslant w  $ & $\iff$ & $Lk(v)$\hspace{-.2cm} & $\subseteq St(w)$,\\
$v \leqslant_s \hspace{-.08cm}w$ & $\iff$ & $St(v)$\hspace{-.2cm} & $\subseteq St(w)$.
\end{tabular}
\end{center}

The preorder $\leqslant$ was already defined in \cite{cv} (see \cite[Lemma 2.1]{cv} for the proof of transitivity of $\leqslant$). 

Below we describe generators of $\OP{Aut}(W_\Gamma)$ \cite{cg}.
They are grouped in four families. 
Since $V$ generates $W_\Gamma$, in the definitions below it amounts to specifying maps from $V$ to $W_\Gamma$. 
In every case it is routine to check that a given map extends to an automorphism. 

\begin{enumerate}
\item Let $\gamma$ be an automorphism of $\Gamma$ such that $\# G_v = \# G_{\gamma(v)}$. 
	The automorphisms of $W$ defined by the permutation of $V$ given by $\gamma$ is called a \textbf{labelled graph automorphism}. 
\item A \textbf{factor automorphism} is an automorphism defined by:
	$$
	\psi_v(z) = 
	\begin{cases}
	v^m & z=v\\
	z & z\neq v,
	\end{cases}
	$$
        where $v \in V$, $m \in \B Z$ and $(m,\# G_v)=1$. Note that if $\#G_v = \infty$, then $m = \pm 1$. 
   
\item 	Let $v,w \in V$ and $v\neq w$. 
	A \textbf{dominated transvection} is an automorphism having one of the following two forms:\\
		$a)$ $\# G_v = \infty$, $v \leqslant w$, and
			$$
			\tau_{v,w}(z) = 
			\begin{cases}
			vw & z=v\\
			\hspace{0.13cm}z & z\neq v,
			\end{cases}
			$$
		$b)$ $\# G_v = p^k$, $\# G_w = p^l$, $v \leqslant_s w$, and
			$$
			\tau_{v,w}(z) = 
			\begin{cases}
			vw^q & z=v\\
			\hspace{0.13cm}z & z\neq v,
        		\end{cases}
			$$
		where $q = max\{1,p^{l-k}\}$ and $p$ is prime. 
\item Let $v \in V$ and $K$ a connected component of $\Gamma - St(v)$. A~\textbf{partial conjugation} is an automorphism defined by:
		$$
		\sigma_{v,K}(z) = 
		\begin{cases}
		vzv^{-1} & z \in K\\
		\phantom{v}z & z\notin K.
		\end{cases}
		$$

\end{enumerate}

The \textbf{pure automorphism group} of $W$, denoted $\OP{\bf Aut}^{\bf 0}(\bf W)$, 
is the subgroup of $\OP{Aut}(W)$ generated by 
factor automorphisms, dominated transvections and partial conjugations. 

\begin{proposition}\label{p:autzero}
$\OP{Aut}^0(W)$ is a normal finite index subgroup in $\OP{Aut}(W)$.
\end{proposition}

\begin{proof}
Let $F$ be the group generated by labelled graph automorphisms. 
It is shown in \cite{cg} that $\OP{Aut}^0(W)$ and $F$ generate $\OP{Aut}(W)$. 
The group $\OP{Aut}^0(W)$ is invariant under conjugations by elements from $F$, thus $\OP{Aut}^0(W)$ is normal. 
Let $F \ltimes \OP{Aut}^0(W)$ be the semidirect product defined by the conjugation action of $F$ on $\OP{Aut}^0(W)$. 
The natural homomorphism $F \ltimes \OP{Aut}^0(W) \to \OP{Aut}(W)$ is onto.
Since $\OP{Aut}^0(W)$ is finite index in $F \ltimes \OP{Aut}^0(W)$, it is so in $\OP{Aut}(W)$. 
\end{proof}

Due to Proposition \ref{p:autzero} and Lemma \ref{l:equivalent} we have.

\begin{corollary}\label{c:eq}
$\OP{Aut(W)}$-invariant word norm and $\OP{Aut}^0(W)$-invariant word norm are equivalent.
\end{corollary}

\section{$\OP{Aut}^0$-invariant subgroups}\label{s:aut^0}


Assume that $G$ is a group, $H \leqslant \OP{Aut}(G)$ and $N \leqslant G$ is an $H$-invariant subgroup. 
Let $p \colon G \to G/N$ be the quotient map and let us define $p_* \colon H \to \OP{Aut}(G/N)$ be $p_*(\psi)(gN) = \psi(g)N$.

Suppose $|\mathord{\cdot}|_i$ is a norm on $G_i$, $i=1,2$.
A map $f \colon G_1 \to G_2$ is Lipschitz with respect to $|\mathord{\cdot}|_1$ and $|\mathord{\cdot}|_2$ if there exists $C \in \B R$ such that
$|f(x)|_2 \leqslant C|x|_1$ for every $x \in G$. 

\begin{lemma}\label{l:lip}
The quotient map $p \colon G \to G/N$ is Lipschitz with respect to $|\mathord{\cdot}|_H$ and $|\mathord{\cdot}|_{p_*(H)}$.
\end{lemma}

\begin{proof}
Let $S \subseteq G$ be a finite set such that $\bar{S} = HS$ generates $G$. 
Thus $p(\bar{S}) = p_*(H)p(S)$ generates $G/N$ 
and we can use $\bar{S}$ to define $|\mathord{\cdot}|_{H}$ and $p(\bar{S})$ to define $|\mathord{\cdot}|_{p_*(H)}$. 
Since $p$ maps generators to generators, we have that $|\mathord{\cdot}|_{p(\bar{S})} \leq |\mathord{\cdot}|_{\bar{S}}$.
\end{proof}

\begin{corollary}\label{c:undistorted}
Let $N \leqslant W_\Gamma$ be an $\OP{Aut}^0(W_\Gamma)$-invariant subgroup and let $p \colon W_\Gamma \to W_\Gamma/N$.
If $p(x) \in W_\Gamma/N$ is $\OP{Aut}$-undistorted, then $x$ is $\OP{Aut}$-undistorted. 
\end{corollary}

\begin{proof}
Let $H = \OP{Aut}^0(W_\Gamma)$. 
We always have $|\mathord{\cdot}|_{Aut(W_\Gamma/N)} \leqslant |\mathord{\cdot}|_{p_*(H)}$.
Due to Lemma \ref{l:lip} and Corollary \ref{c:eq}, $p \colon W_\Gamma \to W_\Gamma/N$ is Lipschitz with respect to 
the $\OP{Aut}$-invariant word norms. The pre-image of an undistorted element by a Lipschitz function is undistorted. 
\end{proof}

Let $W_\Gamma = W(\Gamma,\{G_v\})$ be a graph product of cyclic groups, where each $G_v$ is primary or infinite. 
Given $X \subseteq V$, we define $W_X$ to be the subgroup of $W_\Gamma$ generated by $X$.
The group $W_X$ is called a \textbf{standard subgroup}  of $W_\Gamma$ and is isomorphic to $W(\Gamma(X),\{G_v\}_{v \in X})$,
where $\Gamma(X)$ is the full subgraph of $\Gamma$ spanned by $X$. 

A \textbf{standard retraction} is the map $R_X \colon W_\Gamma \to W_X$ defined by:
	$$
	R_X(z) = 
	\begin{cases}
	z & z \in X\\
	e & z \notin X.
	\end{cases}
	$$
	\vspace{.1cm}
By $K_X$ we denote the kernel of $R_X$.

\begin{lemma}\label{l:factor.partial}
	Let $X \subseteq V$. The group $K_X$ is invariant under factor automorphisms and partial conjugations. 
\end{lemma}

\begin{proof}
	Let $\psi \in \OP{Aut}(W_{\Gamma})$ be a factor automorphism or a partial conjugation. 
	We shall define a map $\psi_0 \in \OP{Aut}(W_X)$ such that the following diagram commutes
$$
\begin{tikzcd}
	W_\Gamma \arrow{r}{R_X} \arrow{d}{\psi} & W_X \arrow{d}{\psi_0}\\
	W_\Gamma \arrow{r}{R_X} & W_X. 
\end{tikzcd}
$$

Then invariance of $K_X$ under $\psi$ follows immediately. 

For $\psi = \phi_v$, a factor automorphism, we define: 
if $v\notin X$, then $\psi_0 = id$; if~$v \in X$, then $\psi_0 = \phi_v \in \OP{Aut}(W_X)$.

For $\psi = \sigma_{v,K}$, a partial conjugation, we define: if $v \notin X$, then $\psi_0 = id$;
if~$v \in X$, then 
	$$
	\psi_0(z) = 
	\begin{cases}
	vzv^{-1} & z \in X\cap K\\
	\phantom{v}z & z \notin X\cap K.
	\end{cases}
	$$

In each case it is clear that $\psi_0 \circ R_X(z) = R_X \circ \psi(z)$ for each $z\in V$, thus the diagram commutes. 
\end{proof}

In general $K_X$ is not invariant under dominated transvections. 
Suppose that there exist $v \notin X$  and $w \in X$ such that $\#G_v=\infty$ and $v \leqslant w$.
Then $\tau_{v,w}$ is well defined and $v \in K_X$, but $\tau_{v,w}(v) = vw \notin K_X$. 
In Lemma \ref{l:lower.cone} we show that existence of such $v$ and $w$ is the only reason why $K_X$ is not invariant under dominated transvections. 

Let $\leqslant_\tau$ be the relation defined on $V$ by 
$$
v \leqslant_\tau w \iff \tau_{v,w}~\text{is well defined,}
$$

i.e.: $v \leqslant_\tau w$ if either a) $\#G_v=\infty$ and $v \leqslant w$ or b) $\#G_v = p^k$, $\#G_w=p^l$ and $v \leqslant_s w$. 

\begin{lemma}
The relation $\leqslant_\tau$ is a partial preorder.
\end{lemma}

\begin{proof}\label{l:taupreorder}
It follows from the fact that $\leqslant$ and $\leqslant_s$ are partial preorders.
\end{proof}

Let $Y$ be a set and $\leqslant$ a relation on $Y$. A subset $X$ of $Y$ is called a~\textbf{lower cone}
if for every $t \in X$ and $s \in Y$ such that
$s \leqslant t$, we have $s \in X$. 

\begin{lemma}\label{l:lower.cone}
If $X \subseteq V$ is a lower cone with respect to $\leqslant_\tau$,
then $K_X$ is $\OP{Aut}^0(W_{\Gamma})$-invariant.
\end{lemma}

\begin{proof}
	Due to Lemma \ref{l:factor.partial} it is enough to show that $K_X$ is invariant under dominated transvections.
	Let $\tau_{v,w} \in \OP{Aut}(W_\Gamma)$ be a dominated transvection. By definition $v \leqslant_\tau w$. 
	We follow the strategy of Lemma \ref{l:factor.partial}.
	If $w \notin X$ then we define $\psi_0 = id$.
	If $w \in X$, then $v \in X$ since $X$ is a lower cone. 
	We define $\psi_0 = \tau_{v,w} \in \OP{Aut}(W_X)$.
\end{proof}

\begin{corollary}\label{c:transfer}
Let $X \subseteq V$ be a lower cone and let $R_X \colon W_\Gamma \to W_X$ be a standard retraction. 
If $R_X(x) \in W_X$ is  $\OP{Aut}$-undistorted, then $x \in W_\Gamma$ is $\OP{Aut}$-undistorted. 
\end{corollary}

\begin{proof}
	$K_X=Ker(R_X)$ is $\OP{Aut}^0(W_\Gamma)$-invariant. We apply Corollary~\ref{c:undistorted}.
\end{proof}
Denote
$$
v \sim_\tau w \iff v \leqslant_\tau w~\text{and}~w \leqslant_\tau v.
$$

Since $\leqslant_\tau$ is a partial preorder, it defines a partial order on equivalence classes of $\sim_\tau$.
We denote this partial order again by $\leqslant_\tau$. 
Equivalence classes of $\sim_\tau$ fall into 3 types which we describe below. 

Let $X \subseteq V$ be an equivalence class of $\sim_\tau$. 

If for some $v \in X$ we have $\#G_v = \infty$, then $\#G_x = \infty$ for every $x \in X$. 
If there exist $v,w \in X$, $v\neq w$ such that $w$ and $v$ commute, then every two elements 
of $X$ commute. In this case $W_X$ is free abelian.
If there exist $v,w \in X$ such that $w$ and $v$ do not commute, then every two elements of $X$ do not commute. 
In this case $W_X$ is a~free group.

If for some $v \in X$ we have $\#G_v = p^k$, then $\#G_x = p^{l_x}$ for some $l_x \in \B N$ for every $x \in X$. 
Note that in this case $W_X$ is always abelian and finite. 

If $X_1$ and $X_2$ are equivalence classes and there exist $v_1 \in X_1$ and $v_2 \in X_2$ such that 
$v_1$ and $v_2$ commute, then every element of $X_1$ commute with every element of $X_2$. In this case we say that $X_1$ and $X_2$ commute. 


\section{Proof of the main theorem}\label{s:proof}


In this section we prove the following theorem which is a generalisation of Theorem \ref{t:main} from the introduction. 

\begin{theorem}\label{t:main2}
Let $\Gamma$ be a finite graph, $V$ its vertex set and let $\{G_v\}_{v\in V}$ be a family of finitely generated abelian groups. 
The $\OP{Aut}$-invariant word norm on $W_\Gamma$ is bounded if and only if $W_\Gamma = Z^n \times D_\infty^m \times F$,
where $n \neq 1$ and $F$ is finite. 
If the $\OP{Aut}$-invariant word norm on $W_\Gamma$ is unbounded, then $W_\Gamma$ has $\OP{Aut}$-undistorted elements.
\end{theorem}

The proof of Theorem \ref{t:main2} is by induction on the number of equivalence classes~of~$\sim_\tau$ and
it is preceded by a number of lemmata.
Lemma \ref{l:dihedral} and Lemma \ref{l:oneclass} are used for the basis of induction.
In Lemma \ref{l:rolli} we introduce quasimorphisms we use in the proof.
Finally, in Lemma \ref{l:free.product} we show the main technical result needed for the inductive step. 

\begin{lemma}\label{l:dihedral}
The $\OP{Aut}$-invariant word norm on the infinite dihedral group $D_\infty$ is bounded. 
\end{lemma}

\begin{proof}
Let $D_\infty = \langle a,b~|~a^2, b^2 \rangle$, $S = \{a,b\}$ and $\bar{S} = \OP{Aut}(D_\infty)S$. 
Let $w \in D_\infty$. Then $w$ is an alternating product of $a$ and $b$. If the length of this product is odd, then 
$w$ is a conjugate of $a$ or $b$, and $w \in \bar{S}$. If the length is even, then $w$ is a product of $a$ or $b$ and 
a conjugate of $a$ or $b$, then $|w|_{\bar{S}} \leqslant 2$. 
\end{proof}

\begin{lemma}\label{l:oneclass}
Let $\Gamma$ be a finite graph, $V$ its vertex set and $\{G_v\}_{v\in V}$ a~family of primary or infinite cyclic groups. 
Assume that $\sim_\tau$ has only one equivalence class. 
Then there are three cases: a) $W_\Gamma = F_n$, a free group,
b) $W_\Gamma = Z^n$, $n>1$, c) $W_\Gamma$ is finite. 
In case a) $W_\Gamma$ has $\OP{Aut}$-undistorted elements, and in cases b) and c) the $\OP{Aut}$-invariant word norm is bounded. 
\end{lemma}

\begin{proof}
	a) $W_\Gamma = F_n$. If $n=1$, then $\OP{Aut}(W_\Gamma) = \{\pm 1\}$ and clearly every non-trivial  element of $Z$ 
is $\OP{Aut}$-undistorted. 
If $n>2$, let $S = \{x\}$ where $x \in F_n$ is a base element. 
The set $\bar{S} = \OP{Aut}(W_\Gamma)S$ is the set of all primitive elements of $F_n$
and it generates $F_n$. 
There exist non-zero homogeneous quasimorphisms 
on $F_n$, which are bounded on $\bar{S}$ \cite{bst}. Thus by Lemma \ref{l:qm.undist}, $F_n$ has $\OP{Aut}$-undistorted elements. 

b) $W_\Gamma = Z^n$, $n>1$. For simplicity suppose $n=2$. If $n>2$ the proof goes along the same lines. 
Let $S~=~\{(1,0)\}$, then 
$$\bar{S} = \OP{Aut}(Z^2)S = \{ (a,b)~|~a,b~\text{are relatively prime} \}.$$

For $(m,n) \in Z^2$, we have $(m,n) = (1,n-1)+(m-1,1)$, thus $|(m,n)|_{\bar{S}}~\leqslant~2$. 
\end{proof}

\begin{lemma}\label{l:rolli}
	Let $G_1$ and $G_2$ be finitely generated groups. 
	Suppose that $G = G_1 \ast G_2$, $G_1 \neq {e}$, $G_2 \neq {e}$  and $G \neq D_\infty$. 
There exists a non-zero homogeneous quasimorphism on $G$ bounded on $G_1\cup G_2$. 
\end{lemma}

\begin{proof}
A function $f \colon G \to \B R$ is odd if $f(x^{-1}) = f(x)^{-1}$ for every $x \in G$. 
Note that if $G \neq C_2^k$, where $C_2$ is the group of order $2$, then there exists a non-zero bounded odd function on $G$. 
Indeed, if $G$ is finitely generated and not isomorphic to $C_2^k$, 
then $G$ has a nontrivial element of order different from $2$. 

Assume now that $G_1$ or $G_2$ is not of the form $C_2^k$. 

Below we describe the construction of so called split quasimorphisms \cite[Section 3.2]{pr}.
Let $\sigma_1 \colon G_1 \to \B R$ and $\sigma_2 \colon G_2 \to \B R$ be bounded odd functions 
such that one of them is non-zero. Let $x \in G$ and let $x~=~x_1x_2\ldots x_n$
be the normal form of $x$ in the free product, i.e. $x_i \in G_{s(i)}$ and $s(i) \neq s(i+1)$. Then 
$$
\sigma_1\ast \sigma_2(x) = \sigma_{s(1)}(x_1) + \ldots + \sigma_{s(n)}(x_n)
$$

is an unbounded quasimorphism. 
It follows that the homogenisation of $\sigma_1 \ast \sigma_2$ is non-zero and is bounded on $G_1\cup G_2$.

If $G = C_2^{k_1} \ast C_2^{k_2}$ and $k_1>1$ or $k_2>1$, then $G$ is non-elementary word hyperbolic, 
and the lemma follows from \cite[Theorem A']{calegari}. 
\end{proof}

\begin{lemma}\label{l:free.product}
Let $\Gamma$ be a finite graph, $V$ its vertex set and $\{G_v\}_{v\in V}$ a family of primary or infinite  cyclic groups. 
Let $M \subseteq V$ be an equivalence class of $\sim_\tau$. 
Suppose $W_\Gamma = W_M \ast (Z^n \times D_\infty^m \times F)$,
where $n \neq 1$ and $F$ is finite. Then $W_\Gamma$ has $\OP{Aut}$-undistorted elements,
provided that $W_\Gamma \neq D_\infty$.
\end{lemma}

\begin{proof}
The class $M$ is minimal with respect to $\leq_\tau$. Indeed, assume by contradiction that $v \in M$, $w~\in~V~-~M$ and $w \leqslant_\tau v$. 
If $\# G_w < \infty$, then $St(w) \subseteq St(v)$ and $w$ commutes with $v$ which leads to a contradiction. 
If $\# G_w = \infty$, then $Lk(w) \subset St(v)$ and since $n>1$, there exists~$w'~\in Lk(w)$ which commutes with $v$, again a contradiction. 

Let $R_M \colon W_\Gamma \to W_M$ be a standard retraction. 
It follows from Corollary \ref{c:transfer} that if $W_M$ has $\OP{Aut}$-undistorted elements, 
then $W_\Gamma$ has $\OP{Aut}$-undistorted elements and the lemma is proven. 

Let us assume that $W_M$ has no $\OP{Aut}$-undistorted elements. 
Since $M$ is an equivalence class, it follows that $W_M = Z^n$, $n>1$ or $W_M$ is finite. 

Now we show that in this case $M$ is also a maximal class. 
Assume by contradiction that $v \in M$ and $w \in V-M$ and $v \leq_\tau w$. 
If $W_M$ is finite, then $St(v) \subseteq St(w)$ and $v$ commutes with $w$ which leads to a contradiction. 
If $W_M = Z^n$, then $Lk(v) \subseteq St(w)$ and there exists~$v'~\in~Lk(v)$ which commutes with $w$, again a contradiction. 

We showed that $M$ is maximal and minimal. Thus if a dominated transvection $\tau_{v,w}$ is well defined, then $v,w~\in~M$ or $v,w~\in~V-M$. 
Let 
$$
(W_M \cup W_{V-M})^{W_\Gamma} = \{ yxy^{-1}~|~y \in W_\Gamma, x \in W_M \cup W_{V-M}\}.
$$

It is clear that  $(W_M \cup W_{V-M})^{W_\Gamma}$ is preserved by dominated transvections, factor automorphisms and partial conjugations.
It follows that $(W_M~\cup~W_{V-M})^{W_\Gamma}$ is an $\OP{Aut}^0(W_\Gamma)$-invariant subset. 

Assume that $W_\Gamma \neq D_\infty$ and let $q$ be a homogeneous quasimorphism from Lemma \ref{l:rolli}, where $G_1 = W_M$ and $G_2 = W_{V-M}$.
Let $S = V$ and $\bar{S} = \OP{Aut}^0(W_\Gamma)S$. 
If $x \in \bar{S}$, then $x = yx_0y^{-1}$, where $x_0 \in W_M \cup W_{V-M}$ and $y \in W_\Gamma$.
We have $q(x) = q(x_0)$, thus $q$ is bounded on $\bar{S}$.
Due to Lemma \ref{l:qm.undist} and Corollary \ref{c:eq}, $W_\Gamma$ has $\OP{Aut}$-undistorted elements. 
\end{proof}

\begin{proof}[Proof of Theorem \ref{t:main2}]
We may assume that for every $v \in V$ the group $G_v$ is primary or infinite cyclic (see Section \ref{s:graph.prod}).

Let us first note, that if $W_\Gamma = Z^n \times D_\infty^m \times F$, $n\neq1$, $F$ finite, then 
the $\OP{Aut}$-invariant word norm on $W_\Gamma$ is bounded. 
Indeed, $|\mathord{\cdot}|_{Aut}$ is bounded on $Z^n$ and $D_\infty$ (Lemma \ref{l:oneclass} and Lemma \ref{l:dihedral})
and the Cartesian product of groups with bounded $\OP{Aut}$-invariant word norms has bounded $\OP{Aut}$-invariant word norm. 

Now we proceed by induction on the number of equivalence classes of~$\sim_\tau$. 
The case when $\Gamma$ has exactly one equivalence class was considered in Lemma \ref{l:oneclass}.

Let $M \subset \Gamma$ be a maximal equivalence class. 
Then $V-M$ is a lower cone. It follows from Corollary \ref{c:transfer} 
that if $W_{M-V}$ has $\OP{Aut}$-undistorted elements, then so does $W_M$.

Let us assume that $W_{M-V}$  has no $\OP{Aut}$-undistorted elements.
Let $\Gamma(V-M)$ be the full subgraph of $\Gamma$ spanned by vertices $V-M$.
The relation $\sim_\tau$ defined with respect to $\Gamma(V-M)$ has fever equivalence classes then
the relation $\sim_\tau$ defined with respect to $\Gamma$. 
Thus, by the induction hypothesis, we have $W_{V-M} = Z^n \times D_\infty^m \times F$, $n \neq 1$, $F$ finite.

Let us remind the reader that if $v \in V-M$ commutes with some $w\in M$, then $v$~commutes with every element of $M$. 
We define 
$$
L =  \{ v \in V-M~|~[v,w]\neq e~\text{ for }~w\in M\}.
$$
If $L$ is empty, then $W_\Gamma = W_M \times W_{V-M}$ and the theorem easily follows. 
Let us assume that $L$ is non-empty.
We have $W_{L} = Z^{n_1} \times D_\infty^{m_1} \times F'$, where $F'$ is finite.
Note that the set of vertices generating $Z^{n_1}$ is a~minimal equivalence class in $V$. 
Thus if $n_1=1$, then
by Corollary \ref{c:transfer}, $W_\Gamma$ has $\OP{Aut}$-undistorted elements. 
Let us assume that $n_1 \neq 1$. 

The set $M \cup L$ is a lower cone in $\Gamma$.
Indeed, if $v \in L$ and $w \leqslant_\tau v$, then $Lk(w) \subset St(v) \subset V-M$, thus $w \in L$. 
Hence $W_\Gamma$ has $\OP{Aut}$-undistorted elements provided that $W_M \neq C_2$ 
or $W_L \neq C_2$ (Lemma \ref{l:free.product} and Corollary \ref{c:transfer}). 

There is one last case to consider, namely $W_M = W_L = C_2$.
Let $v_0 \in V$ be the generator of $W_L$ and $w$ the generator of $W_M$.
Note that $v_0 \in F$. 
Indeed, otherwise $v_0$ would be a generator of a $D_\infty$ factor of $W_{V-M}$. 
Then the other generator of $D_\infty$, say $v_1$, would commute with $w$ and then $St(w) \subseteq St(v_1)$
which gives $w \leq_\tau v_1$. This contradicts maximality of $w$.
Thus $v_0 \in F$ and 
$$
W_\Gamma = Z^n \times D_\infty^m \times \langle w,v_0 \rangle \times F/\langle v_0 \rangle = 
Z^n \times D_\infty^{m+1} \times F/\langle v_0 \rangle,
$$
a group with bounded $\OP{Aut}$-invariant word norm.
\end{proof}
 
\section{Problems}

\begin{problem}\label{p1}
For which $W_\Gamma$ does there exist a non-zero homogeneous $\OP{Aut}$-invariant (or $\OP{Aut}^0(W_\Gamma)$-invariant) quasimorphism on $W_\Gamma$?
\end{problem}

If there exists a non-zero homogeneous $\OP{Aut}(G)$-invariant quasimorphism on $G$, 
then the $\OP{Aut}$-invariant word norm is automatically unbounded. 
For $W_\Gamma = F_n$, Problem \ref{p1} was posed by Mikl\'{o}s Ab\'{e}rt \cite[Question 47]{abert}. 
The answer is positive if $W_\Gamma$ is the free group of rank two \cite[Theorem 2]{bm}. 

We say that $x \in G$ is $\OP{\bf Aut}$-\textbf{bounded}, if there exists $C \in \B R$ such that $|x^n|_{Aut}<C$ for every $n \in \B Z$.

\begin{problem}\label{p2}
Is every $\OP{Aut}$-distorted element of $W_\Gamma$ also $\OP{Aut}$-bounded?
If $x \in W_\Gamma$ is $\OP{Aut}$-undistorted, does there exist a homogeneous quasimorphism $q \colon W_\Gamma \to \B R$
which is Lipschitz with respect to the $\OP{Aut}$-invariant word norm and such that $q(x) \neq 0$?
\end{problem}

\begin{problem}
Characterise $\OP{Aut}$-distorted and $\OP{Aut}$-undistorted elements of $W_\Gamma$. 
\end{problem}

For free groups, Theorem 2.10 in \cite{bm} provides a simple characterisation of $\OP{Aut}$-undistorted (or bounded) elements.

\bibliography{bibliography}
\bibliographystyle{plain}

\end{document}